\documentclass[final]{amsart}
\usepackage[utf8]{inputenc}
\usepackage{tikz-cd}
\usepackage{nicefrac}
\usepackage{enumerate}
\usepackage{amsmath,amsbsy,amssymb}
\usepackage{mathtools}
\mathtoolsset{showonlyrefs}
  
\providecommand{\keywords}[1]
{
  \small	
  \textbf{\textit{Keywords---}} #1
}

\usepackage{todonotes}
\usepackage[normalem]{ulem}
\usepackage[notref,notcite]{showkeys}
\usepackage[mode=multiuser]{fixme}
\FXRegisterAuthor{gk}{envgk}{GK}
\FXRegisterAuthor{fb}{envfb}{FB}
\fxusetheme{color}
\definecolor{fxnote}{rgb}{0.0000,0.0000,1.0000}
\definecolor{fxwarning}{rgb}{.2000,0.60,0.0000}
\definecolor{fxerror}{rgb}{.900,0.4706,0.0000}
\definecolor{fxfatal}{rgb}{1.000,0.0000,0.0000}
\usepackage{wasysym}
\usepackage{cancel}

\newtheorem{theorem}{Theorem}[section]
\newtheorem{lemma}[theorem]{Lemma}

\theoremstyle{definition}
\newtheorem{definition}[theorem]{Definition}

\theoremstyle{remark}
\newtheorem{remark}[theorem]{Remark}

\numberwithin{equation}{section}

\let\rho\varrho
\let\phi\varphi
\let\epsilon\varepsilon

\newcommand{\dx}{\,dx}
\newcommand{\dtx}{\,d\tilde x}

\renewcommand{\P}{\mathbb P}
\newcommand{\R}{\mathbb R}

\newcommand{\abs}[1]{\left|#1\right|}
\newcommand{\norm}[2]{\left\|#1\right\|_{#2}}
\newcommand{\bc}[2]{\left(\begin{array}{c} #1\\#2\end{array}\right)}

\def\interval{\mathcal I}

\def\nodal{\mathcal N}

\newcommand{\cint}{I}
\newcommand{\wint}{\Pi}

\newcommand{\winttens}[2]{\Pi^{\otimes #1}_{#2}}

\newcommand\Alt{\operatorname{Alt}}

\newcommand{\complextens}[3]{(#1\Lambda^{\otimes #2})^{#3}}

\newcommand{\complexPtens}[2]{(\P\Lambda^{\otimes #1})^{#2}}

\newcommand{\complexLtens}[2]{(L^2\Lambda^{\otimes #1})^{#2}}
\newcommand{\complexH}[1]{\left(H\Lambda\otimes H\Lambda\right)^{#1}}
\newcommand{\complexHtens}[2]{(H\Lambda^{\otimes #1})^{#2}}

\newcommand{\dual}[3]{\left(#1,#2\right)_{#3}}

\newcommand{\range}{\operatorname{Im}}
\newcommand{\kernel}{\operatorname{Ker}}

\begin{document}

\title[$H^1$-conforming finite element cochain complex]{$H^1$-conforming finite element cochain complexes and commuting quasi-interpolation operators on Cartesian meshes}

\author{Francesca Bonizzoni}
\address{Faculty of Mathematics, University of Vienna, Oskar-Morgenstern-Platz 1,1090 Wien, Austria}
\email{francesca.bonizzoni@univie.ac.at}
\thanks{F. Bonizzoni acknowledges partial support from the Austrian Science Fund (FWF) through the project F 65, and has been supported by the FWF Firnberg-Program, grant T998.}

\author{Guido Kanschat}
\address{Interdisciplinary Center for Scientific Computing (IWR), Heidelberg University, Klaus-Tschira-Platz 1, 69120 Heidelberg, Germany}
\curraddr{}
\email{kanschat@uni-heidelberg.de}
\thanks{The authors acknowledge support by the Erwin Schrödinger Institute (ESI) at University of Vienna through the Thematic Programme ``Numerical Analysis of Complex PDE Models in the Sciences''.}



\date{}

\dedicatory{}

\begin{abstract}
  A finite element cochain complex on Cartesian meshes of any dimension based on the $H^1$-inner product is introduced. It yields $H^1$-conforming finite element spaces with exterior derivatives in $H^1$. We use a tensor product construction to obtain $L^2$-stable projectors into these spaces which commute with the exterior derivative. The finite element complex is generalized to a family of arbitrary order.
\end{abstract}


\maketitle

\section{Introduction}

We present a family of finite element cochain complexes in $H^1(\Omega)$ on Cartesian meshes. By adhering to a strict tensor product construction, we obtain commuting interpolation operators which are bounded on $L^2(\Omega)$.

It has been pointed out for instance in~\cite{SchroederLehrenfeldLinkeLube18} that the reliable computation of high Reynolds number incompressible flow hinges on pressure robustness of the discretization, which in turn is guaranteed by exact implementation of the divergence condition. We have demonstrated the importance of the cochain property for error estimates and adaptive mesh refinement in~\cite{KanschatSharma14,SharmaKanschat18}. These works have in common that they rely on divergence-conforming discontinuous Galerkin methods, which started with~\cite{HansboLarson02,CockburnKanschatSchoetzau07}. Thus, they are consistent with the Laplacian, but not conforming in $H^1(\Omega)$.

Due to the importance of the divergence constraint, considerable effort was put into the development of $H^1$-conforming methods with exact divergence constraint in recent years. In particular on simplicial meshes, there is a wide variety of methods. We refer the reader to the recent review~\cite{Neilan20} and the literature cited therein.

The use of the Raviart-Thomas polynomial space $\mathbb Q_{3,2}\times \mathbb Q_{2,3}$ with node functionals yielding $H^1$-conforming finite elements on rectangular meshes goes back to~\cite{AustinManteuffelMcCormick04}. They already use a tensor product of Hermitian and Lagrangian interpolation on each rectangle, such that the divergence is in the space of continuous functions and cellwise in $\mathbb Q_{2,2}$.

The same polynomial space for the velocity, but with different degrees of freedom is used in~\cite{Zhang09}, but with an implicitly defined pressure space. The author obtains a solution by a procedure which does not require setting up a basis for the pressure space, and thus the construction is valid. Nevertheless, the discretization spaces are bound to a specific solution scheme for the discrete problem. This was overcome later in~\cite{NeilanSap16} by using partly Hermitian interpolation, thus obtaining a local characterization of the pressure space. They use Hermitian degrees of freedom in vertices, but only Lagrangian on edges, such that the pressure space can be discontinuous. As a result, the velocity space does not result from tensorization of one-dimensional elements, which is one of our construction principles.
Also the inf-sup stable Stokes pair of finite elements yielding diverge-free solutions in~\cite{NeilanSap18} does not have tensor product structure.

Discretization spaces which are $H^1$-conforming and yield diverge-free solutions have been object of the Isogeometric Analysis (IGA) literature, too. In this framework, tensor product meshes and spline-based approximation spaces are considered. We refer to~\cite{BuffaDeFalcoSangalli11} for IGA techniques applied to the Stokes problem, and to~\cite{EvansHughes13,EvansHughesUnstaedy13} for applications to the steady and unsteady Navier-Stokes equations.

Quasi-interpolation operators for the element from~\cite{AustinManteuffelMcCormick04} which commute with the divergence were first introduced in~\cite{SharmaKanschat18}.
Here, we systematically reconstruct the canonical interpolation operators used there and generalize them to any space dimension and forms of any index.

While all these publications were concerned with $H^1$-conforming elements with controllable divergence in two and three dimensions, this paper is concerned with the full finite element cochain complex on Cartesian meshes of arbitrary dimension, such that for each finite element form $u_h\in V^k \subset H^1\Lambda^k(\Omega) $ its exterior derivative is $d u_h \in V^{k+1} \subset H^1\Lambda^{k+1}(\Omega)$. Based on a general lemma on the cochain property of interpolation operators, we provide commuting interpolation operators for differentiable functions as well as commuting quasi-interpolation operators which are continuous on $L^2$.
For the latter, we follow the route laid out in~\cite{Schoeberl01,Schoeberl,Schoeberl10} in one dimension and tensorize afterwards.

This article is laid out as follows: after some preliminaries in section~\ref{sec:notation} we present a general construction principle for commuting interpolation operators in section \ref{sec:construction}. The one-dimensional finite element cochain complex based on cubic polynomials with Hermitian interpolation is outlined in section~\ref{sec:1d_complex_I} and its quasi-interpolation operators are introduced in section~\ref{sec:quasi_interpolation}. The tensorization for higher-dimensional complexes is presented in sections~\ref{sec:tensor_complex} and~\ref{sec:tensor-quasi-interpolation}, respectively. Section~\ref{sec:higher-order} presents the extension to higher order polynomial spaces.

\section{Notation and preliminaries}
\label{sec:notation}

Following~\cite{ArnoldFalkWinther2006}, we introduce the notation and definitions concerning the finite element exterior calculus that we will need throughout the paper. Let $n\geq 1$ and $0\leq k\leq n$ integers. We denote with $\Alt^k\R^n$ the space of alternating $k$-linear forms on $\R^n$, with inner product $\dual{\cdot}{\cdot}{\Alt^k\R^n}$.
Let $\Omega$ be a $n$-dimensional open bounded 
subset of $\R^n$. A differential $k$-form on $\Omega$ is a map $u$ which associates to each $x\in\Omega$ an element $u_x\in\Alt^k\R^n$. It can be expressed uniquely as
\begin{equation}
    \label{eq:u_form}
    u=\sum_{\sigma\in\Sigma(k,n)} u_\sigma\, dx^\sigma,
\end{equation}
where $u_\sigma$ are coefficient functions defined on $\Omega$, and $\Sigma(k,n)$ is the set of increasing maps $\{1,\ldots,k\}\rightarrow\{1,\ldots,n\}$. The set $\{dx^1,\ldots,dx^n\}$ denotes the basis of $\Alt^1\R^n=(\R^n)^*$ dual to the canonical basis, and $dx^\sigma=dx^{\sigma_1}\wedge\cdots\wedge dx^{\sigma_k}\in \Alt^k\R^n$.
We denote with $\Lambda^k(\Omega)$ the space of smooth differential $k$-forms, i.e., the space of $k$-forms with smooth coefficient functions.

Let $d^k:\Lambda^k(\Omega)\rightarrow\Lambda^{k+1}(\Omega)$ be the exterior derivative, i.e, the linear map which associates $u\in\Lambda^k(\Omega)$ as in~\eqref{eq:u_form} to $d^k u\in\Lambda^{k+1}(\Omega)$ given by
\begin{equation*}
    d^k u = \sum_{\sigma\in\Sigma(k,n)}\sum_{j=1}^n
    \frac{\partial u_\sigma}{\partial x_j}\ dx^j\wedge dx^\sigma.
\end{equation*} 
In the following, when no confusion occurs, we will denote the exterior derivative simply as $d$, suppressing the superscript $k$.

Given $\mathcal F(\Omega)$ a space of functions defined on $\Omega$, we denote with $\mathcal F\Lambda^k(\Omega)$ the space of differential $k$-forms with coefficients in $\mathcal F(\Omega)$. 
As examples, we mention 
the space of $C^m$-regular differential $k$-forms $C^m\Lambda^k(\Omega)$, the space of $L^2(\Omega)$-integrable $k$-forms $L^2\Lambda^k(\Omega)$, and the space of polynomial differential $k$-forms $\P_m\Lambda^k(\Omega)$.

The space $L^2\Lambda^k(\Omega)$ is a Hilbert space, with inner product
\begin{align}
\label{eq:L2_inner_prod}
    \dual{\cdot}{\cdot}{L^2\Lambda^k}&:L^2\Lambda^k(\Omega)\times L^2\Lambda^k(\Omega)\rightarrow\R\\
    \nonumber
    \dual{u}{v}{L^2\Lambda^k}&=\sum_{\sigma\in\Sigma(k,n)}\dual{u_\sigma}{v_\sigma}{L^2(\Omega)}.
\end{align}
We define the space $H\Lambda^k(\Omega)$ as 
\begin{equation}
    \label{eq:Hlambda}
    H\Lambda^k(\Omega)
    :=\left\{
    u\in L^2\Lambda^k(\Omega)\, |\, du\in L^2\Lambda^{k+1}(\Omega)
    \right\}.
\end{equation}
It is a Hilbert space, with the inner product
\begin{align}
\label{eq:H_inner_prod}
    \dual{\cdot}{\cdot}{H\Lambda^k}&:H\Lambda^k(\Omega)\times H\Lambda^k(\Omega)\rightarrow\R\\
    \nonumber
    \dual{u}{v}{H\Lambda^k}&=\dual{u}{v}{L^2\Lambda^k}+\dual{du}{dv}{L^2\Lambda^{k+1}}.
\end{align}

The extended de Rham complex is the following sequence of spaces and maps:
\begin{gather}
\label{eq:exact_Lk}
    0\xrightarrow{\;\;\subset\;\;}\R\xrightarrow{\;\;\subset\;\;}
    \Lambda^0(\Omega) \xrightarrow{\;\; d \;\;} \Lambda^1(\Omega)\xrightarrow{\;\;d\;\;}\cdots
    \xrightarrow{\;\; d\;\;}
    \Lambda^n(\Omega)\xrightarrow {\;\; d\;\;} 0.
\end{gather}

From the relation $d\circ d=0$, it follows that
\begin{equation*}
    \range\left(d^{k-1}\right)\subset \kernel\left(d^{k}\right),    
\end{equation*}
where $\range$ and $\kernel$ denote the range and the kernel, respectively.
In the case of a contractible domain, the sequence~\eqref{eq:exact_Lk} is exact, meaning that 
\begin{equation*}
    \range\left(d^{k-1}\right)= \kernel\left(d^{k}\right), 
\end{equation*}
In the case of a noncontractible domain, the codimension of $\range(d)$ in $\kernel(d)$ is equal to the corresponding Betti number.

The complex
\begin{gather}
    \label{eq:complex_H}
    0\xrightarrow{\;\;\subset\;\;}\R
    \xrightarrow{\;\;\subset\;\;} H\Lambda^0(\Omega)
    \xrightarrow{\;\; d\;\;} H\Lambda^1(\Omega)
    \xrightarrow{\;\;d\;\;}\cdots
    \xrightarrow{\;\; d\;\;} H\Lambda^n(\Omega)
    \xrightarrow{\;\;d\;\;}0
\end{gather}
is the $L^2$ de Rham complex on $\Omega$.

In the same way, given $m\geq n$ integer, it is possible to construct a cochain complex with the set of spaces $\{C^{m-k}\Lambda^k(\Omega),\ k=0,\ldots,n\}$, as follows
\begin{gather}
    \label{eq:nD_complex}
    0\xrightarrow{\;\;\subset\;\;}
    \R\xrightarrow{\;\;\subset\;\;}
    C^m\Lambda^0(\Omega)
    \xrightarrow{\;\; d\;\;}
    C^{m-1}\Lambda^1(\Omega)
    \xrightarrow{\;\;d\;\;}\cdots
    C^{m-n}\Lambda^n(\Omega)\xrightarrow{\;\; d\;\;} 0.
\end{gather}

\section{Construction of a commuting interpolation operator}
\label{sec:construction}

Let $0\leq k<n$ be fixed, and let $\P\Lambda^k(\Omega)$ and $\P\Lambda^{k+1}(\Omega)$ be polynomial forms such that $d\P\Lambda^k(\Omega) \subset \P\Lambda^{k+1}(\Omega)$.
Denote with $r$ the dimension of the range of $d$, that is, by the rank-nullity theorem, $r = \dim\P\Lambda^k-\dim\kernel d$.
Moreover, let $\cint_k\colon \Lambda^k(\Omega)\rightarrow\P\Lambda^k(\Omega)$ and $\cint_{k+1}\colon \Lambda^{k+1}(\Omega)\rightarrow\P\Lambda^{k+1}(\Omega)$
denote the interpolation operators defined by
\begin{gather}
\label{eq:cint}
    \cint_j u = \sum_{i=1}^{\dim\P\Lambda^j} \nodal^j_i(u)\phi^j_i,\qquad j=k,k+1,
\end{gather}
where 
\begin{itemize}
\item $\{\phi^j_i\}_{i=1,\ldots,\dim\P\Lambda^j(\Omega)}$ is a basis for $\P\Lambda^j(\Omega)$;
\item $\{\nodal^j_i\}_{i=1,\ldots,\dim\P\Lambda^j(\Omega)}$ is a set of node functionals in $(\Lambda^j(\Omega))^*$;
\item both sets are chosen such that there holds the interpolation condition
\begin{gather}
    \label{eq:interpolation-condition}
    \nodal^j_i(\phi^j_m) = \delta_{im}.
\end{gather}
\end{itemize}
The following lemma gives sufficient conditions on $\{\phi^j_i\}_i$ and $\{\nodal^j_i\}_i$ such that 
\begin{equation*}
\begin{tikzcd}[column sep=3pc]
\Lambda^k(\Omega) \arrow{r}{d} \arrow{d}{I} & 
  \Lambda^{k+1}(\Omega) \arrow{d}{I} \\
\P\Lambda^k(\Omega) \arrow{r}{d} &
  \P\Lambda^{k+1}(\Omega)
\end{tikzcd}
\end{equation*}
is a commuting diagram. Here we used the convention that $I=I_k$ if its domain is $\Lambda^k$.
It simplifies the construction of commuting quasi-in\-ter\-po\-lation operators, since it later on only requires the transformation of node functionals to perturbed cells, not the transformation of basis functions.

\begin{lemma}
\label{lemma:commute_nodal}
With the notations introduced above,
assume that the bases $\{\phi^k_i\}$ and $\{\phi^{k+1}_i\}$ fulfill
  \begin{gather}
    \label{eq:commute-basis}
      \begin{aligned}
          d\phi^k_i &= \phi^{k+1}_i & \qquad i&=1,\dots,r,\\
          d\phi^k_i &= 0 & i&=r+1,\dots,\dim\P\Lambda^{k}(\Omega).
      \end{aligned}
  \end{gather}
Moreover, assume that the node functionals $\{\nodal^k_i\}$ and $\{\nodal^{k+1}_i\}$ fulfill for any $u\in \Lambda^k(\Omega)$
  \begin{gather}
    \label{eq:commute-nodal}
      \begin{aligned}
      \nodal^{k+1}_i(d u) &= \nodal^k_i(u)& \qquad i&=1,\dots,r,\\
      \nodal^{k+1}_i(d u) &= 0 & i&=r+1,\dots,\dim\P\Lambda^{k+1}.
      \end{aligned}
  \end{gather}
   Then, the interpolation operator $\cint$ defined in~\eqref{eq:cint} commutes with the exterior derivative $d$, namely, there holds:
  \begin{gather}
    d^k \cint_k u = \cint_{k+1} d^k u \qquad \forall u \in \Lambda^k(\Omega).
  \end{gather}
\end{lemma}

\begin{proof}
  By linearity, we have
  \begin{gather}
      d^k\cint_k u = d^k\left(\sum_{i=1}^{\dim\P\Lambda^k}\nodal^k_i(u) \phi^k_i\right)
      = \sum_{i=1}^{\dim\P\Lambda^k}\nodal^k_i(u) d^k \phi^k_i
      = \sum_{i=1}^{r}\nodal^k_i(u) \phi^{k+1}_i.
  \end{gather}
  On the other hand,
  \begin{gather}
      \cint_{k+1} d^k u = \sum_{i=1}^{\dim\P\Lambda^{k+1}}\nodal^{k+1}_i(d^k u) \phi^{k+1}_i
      = \sum_{i=1}^{r}\nodal^{k+1}_i(d^ku) \phi^{k+1}_i.
  \end{gather}
  Employing~\eqref{eq:commute-nodal} concludes the proof.
\end{proof}

The lemma states, that we can construct commuting interpolation operators in five steps:
\begin{enumerate}
    \item Choose node functionals for $\range d \subset \P\Lambda^{k+1}$.
    \item Choose node functionals for $\P\Lambda^k$ according to~\eqref{eq:commute-nodal}.
    \item Choose a basis for $\range d \subset \P\Lambda^{k+1}$ such that interpolation condition~\eqref{eq:interpolation-condition} holds.
    \item Choose a basis for $\P\Lambda^k$ according to~\eqref{eq:interpolation-condition} and~\eqref{eq:commute-basis}.
    \item Choose the remaining basis functions and node functionals such that~\eqref{eq:commute-basis} and~\eqref{eq:commute-nodal} hold.
\end{enumerate}

\begin{remark}
\label{rem:Cm}
Lemma~\ref{lemma:commute_nodal} applies naturally to 
\begin{gather}
    \begin{aligned}
        \cint_k&\colon& C^{m-k}\Lambda^k(\Omega)&\rightarrow\P\Lambda^k(\Omega)\\
        \cint_{k+1}&\colon& C^{m-k-1}\Lambda^{k+1}(\Omega)&\rightarrow\P\Lambda^{k+1}(\Omega)
    \end{aligned}
\end{gather}
for all $m\geq k+1$, entailing the commutativity of the following diagram
\begin{equation*}
\begin{tikzcd}[column sep=3pc]
C^{m-k}\Lambda^k(\Omega) \arrow{r}{d} \arrow{d}{I} & 
  C^{m-k-1}\Lambda^{k+1}(\Omega) \arrow{d}{I} \\
\P\Lambda^k(\Omega) \arrow{r}{d} &
  \P\Lambda^{k+1}(\Omega)
\end{tikzcd}
\end{equation*}
\end{remark}

\section{The one-dimensional complex on the reference interval}
\label{sec:1d_complex_I}

Within this section we take $\Omega$ equal to the unit interval $\interval=[0,1]$, and $m=1$. Then the cochain complex~\eqref{eq:nD_complex} becomes the exact sequence
\begin{gather}
    \label{eq:1D_complex}
    0 \xrightarrow{\;\;\subset \;\;}\R
    \xrightarrow{\;\; \subset \;\;}
    C^1\Lambda^0(\interval)
    \xrightarrow{\;\; d \;\;}
    C^0\Lambda^1(\interval)
    \xrightarrow{\;\; d \;\;}
    0.
\end{gather}

We discretize the spaces $C^1\Lambda^0(\interval)$ and $C^0\Lambda^1(\interval)$ by polynomial spaces
\begin{gather}
    \P_3\Lambda^0(\interval) = \P_3(\interval)
    \qquad\text{and}\qquad
    \P_2\Lambda^1(\interval) = \P_2(\interval),
\end{gather}
respectively. In particular, we consider 
\begin{itemize}
    \item the space $\P^3\Lambda^0(\interval)=\P_3(\interval)$ with the interpolation operator $\cint_0\colon C^1\Lambda^0(\interval)\to\P_3\Lambda^0(\interval)$ defined by the conditions $\nodal^0_i \cint_0 u = \nodal^0_i u$, for $i=1,2,3,4$, where the node functionals $\nodal^0_i:C^1\Lambda^0(\interval)\rightarrow\R$ are given by a modified Hermitian interpolation:
    \begin{gather}
        \label{eq:i0-1d}
        \begin{aligned}
        \nodal^0_1 (u) &= u'(0) \qquad& \nodal^0_3 (u) &= u(1)-u(0)\\
        \nodal^0_2 (u) &= u'(1) & \nodal^0_4 (u) &= u(1)+u(0)
        \end{aligned}
    \end{gather}
    \item the space $\P_2\Lambda^1(\interval) = \P_2(\interval)$ with the interpolation operator $\cint_1\colon C^0\Lambda^1(\interval)\to\P_2\Lambda^1(\interval)$ defined by the conditions $\nodal^1_j \cint_1 v = \nodal^1_j v$, for $j=1,2,3$, where the node functionals $\nodal^1_j:C^0\Lambda^1(\interval)\rightarrow\R$ are given by:
    \begin{gather}
        \label{eq:i1-1d}
        \begin{split}
            \nodal^1_1 (v) &= v(0)\\
            \nodal^1_2 (v) &= v(1)
        \end{split}
        \qquad \nodal^1_3 (v) = \int_{\interval} v(x)\dx.
    \end{gather}
\end{itemize}

The second set of node functionals is a well-known alternative to Lagrange interpolation and it yields a unisolvent finite element. The basis of $\P_2(\interval)$ dual to the set of node functionals in~\eqref{eq:i1-1d} is:
\begin{gather}
    \label{eq:basis-p2}
    \phi^1_1(x)=1-4x+3x^2,\quad\phi^1_2(x)=-2x+3x^2,\quad\phi^1_3(x)=6x-6x^2.
\end{gather}

By straight forward computation, we obtain the following result.
\begin{lemma}
  If in $\P_3(\interval)$ we choose the basis
  \begin{gather}
    \label{eq:basis-p3}
      \begin{aligned}
      \phi^0_1(x) &= x-2x^2+x^3 \qquad& \phi^0_3(x) &= -\tfrac{1}{2}+3x^2-2x^3\\
      \phi^0_2(x) &= -x^2+x^3 & \phi^0_4(x) &= \tfrac12
      \end{aligned}
  \end{gather}
  there holds $\nodal^0_i(\phi^0_j) = \delta_{ij}$, and the element is unisolvent.
\end{lemma}

In this framework, the interpolation operators defined in~\eqref{eq:cint} become:
\begin{gather}
    \label{eq:cint-1d}
    \cint_0 u(x) = \sum_{i=1}^4 \nodal^0_i(u) \,\phi^0_i(x),
    \qquad \cint_1 v(x) = \sum_{j=1}^3 \nodal^1_j (v) \,\phi^1_j(x).
\end{gather}
The following lemma shows that $I_0$ and $I_1$ commute with the exterior derivative.

\begin{lemma}
\label{lemma:commuting-1d}
The following diagram commutes:
\begin{equation}
\label{eq:diagram}
\begin{tikzcd}[column sep=6pc]
C^1\Lambda^0(\interval) \arrow{r}{d} \arrow{d}{I} & 
  C^0\Lambda^1(\interval) \arrow{d}{I} \\
\P_3\Lambda^0(\interval) \arrow{r}{d} &
  \P_2\Lambda^1(\interval)
\end{tikzcd}
\end{equation} that is,
for every $u\in C^1\Lambda^0(\interval)$, there holds
\begin{gather}
    \label{eq:commute-interp-1d}
    d^0 \cint_0 u = \cint_1 d^0 u.
\end{gather}
\end{lemma}

\begin{proof}
To prove the result, it is enough to verify that the assumptions of lemma~\ref{lemma:commute_nodal} are fulfilled.
Comparing the bases in~\eqref{eq:basis-p3} and~\eqref{eq:basis-p2}, we see by straight forward computation that~\eqref{eq:commute-basis} holds. 
In particular, $d\phi^0_4 \equiv 0$, such that $\phi^0_4$ spans the kernel of $d$. Moreover,
\begin{gather}
    \nodal^1_1(du) = du(0) = u'(0) = \nodal^0_1(u).
\end{gather}
The same argument yields for $\nodal^1_2(du) = \nodal^0_2(u)$.
Furthermore,
\begin{gather}
    \label{eq:n31=n30}
    \nodal^1_3(du) = \int_{\interval} du \dx= \int_{\interval} u' \dx
    = u(1) - u(0) = \nodal^0_3(u).
\end{gather}
\end{proof}

\begin{remark}
\label{rem:continuity}
The set of node functionals $\{\nodal^0_j\}_{j=1}^4$ doesn't contain the evaluation functionals in the end points of the interval $\interval$ individually. Hence it is not immediately obvious that the interpolation operator $\cint_0$ generates continuity. Nevertheless, a simple computation shows that $\cint_0u(x)=u(x)$ for $x=0,1$. 
\end{remark}

\section{Quasi-interpolation operators}
\label{sec:quasi_interpolation}

The node functionals introduced in Section~\ref{sec:1d_complex_I} require point values of first derivative of $u$. In this section we want to weaken this condition, by defining weighted node functionals, which yield quasi-interpolation operators on $L^2$.

\subsection{Node functionals on perturbed intervals}

We begin by introducing perturbations of the reference interval $\interval=[0,1]$.
Let 
$0<\rho\leq\frac{1}{3}$ be a fixed parameter and $\interval_\rho = [-\rho, 1+\rho]$. Choose $y_l,\,y_r\in\interval_\rho$ such that $y_l\in B_\rho(0)$ and $y_r\in B_\rho(1)$, where $B_\rho(x)$ denotes the interval of radius $\rho$ and center $x$.
The perturbed interval $\tilde\interval_{y_l,y_r} = [y_l, y_r]$ is defined as the image of the reference interval $\interval$ via a monotone, possibly non-linear mapping
\begin{equation}
    \label{eq:phi}
    \Phi_{y_l,y_r}: \interval\to \tilde\interval_{y_l,y_r}
\end{equation}
satisfying $\Phi_{y_l,y_r}^{\prime}\equiv 1$ on $B_\rho(0)$ and $B_\rho(1)$.
It turns out that the actual shape of $\Phi$ is not needed in the definition of the quasi-interpolation operators.
We transform the node functionals for $\P_3\Lambda^0$ in~\eqref{eq:i0-1d} as:
\begin{gather}
    \label{eq:node-tilde}
    \begin{aligned}
    \widetilde{\nodal^0_1} (u) &= u'(y_l) \qquad
    & \widetilde{\nodal^0_3} (u) &= u(y_r)-u(y_l)\\
    \widetilde{\nodal^0_2} (u) &= u'(y_r)
    & \widetilde{\nodal^0_4} (u) &= u(y_r)+u(y_l)
    \end{aligned}
\end{gather}
and those for $\P_2\Lambda^1$ in~\eqref{eq:i1-1d} as:
\begin{gather}
    \label{eq:node1-tilde}
        \begin{split}
            \widetilde{\nodal^1_1} (v) &= v(y_l)\\
            \widetilde{\nodal^1_2} (v) &= v(y_r)
        \end{split}
        \qquad
    \widetilde{\nodal^0_3}(v)
    = \int_{\tilde\interval_{y_l,y_r}} v(\tilde x)
    \dtx.
\end{gather}

By proceeding as in the proof of Lemma~\ref{lemma:commuting-1d}, we obtain the following result.
\begin{lemma}
\label{prop:dof_tilde_prop}
The transformed node functionals~\eqref{eq:node-tilde} and~\eqref{eq:node1-tilde} satisfy the following property:
\begin{align}
    \label{eq:dof_tilde_prop1}
    \widetilde{\nodal^1_i}(du)&=\widetilde{\nodal^0_i}(u), \quad i=1,2,3,
\end{align}
for all $u\in C^1\Lambda^0(\interval_\rho)$.
\end{lemma}

\subsection{Weighted node functionals}

Let $\eta\in C^\infty(\mathbb R)$ be the standard mollifier
\begin{gather*}
    \eta(x):=\left\{\begin{array}{ll}
        C\, \exp\left(\frac{1}{|x|^2-1}\right), & \text{if }|x|<1\\
        0, &\text{if }|x|>1
    \end{array}\right.    
\end{gather*}
where the constant $C$ normalizes the integral to 1. Then,
the cut-off functions for the intervals
$B_l:=B_\rho(0)$ and $B_r:=B_\rho(1)$ are given by
\begin{gather*}
    \eta_l(x):=\frac{1}{\rho}\eta\left(\frac{x}{\rho}\right),
    \quad
    \eta_r(x):=\frac{1}{\rho}\eta\left(\frac{x-1}{\rho}\right).
\end{gather*}
Note that, due to normalization, it holds
\begin{gather}
    \label{eq:normalization}
    \int_{\mathbb R} \eta_l(x) \,dx = \int_{\mathbb R} \eta_r(x) \,dx = 1,
    \qquad \norm{\eta_l}{L^2(\mathbb R)}
    = \norm{\eta_r}{L^2(\mathbb R)}
    = \frac{\norm{\eta}{L^2(\mathbb R)}}{\sqrt\rho}.
\end{gather}

We introduce the weighted node functionals $\overline{\nodal^k_i}\in(C^{1-k}\Lambda^k(\interval_\rho))^*$ as follows:
\begin{gather}
    \label{eq:pi0-1d}
    \overline{\nodal^k_i}(u) = \iint_{B_l\, B_r} \eta_{l}(\xi_l)\eta_{r}(\xi_r)
    \widetilde{\nodal^k_i}(u) \, d\xi_r \, d\xi_l, 
\end{gather}
for $k=0,1$ and all admissible values of $i$.

\begin{remark}
The normalization entails, for instance,
\begin{gather*}
    \overline{\nodal^0_1}(u)
    =\int_{B_l} \eta_{l}(\xi)
    \widetilde{\nodal^0_1}(u)  \, d\xi.
\end{gather*}
Thus, the weighted node functionals $ \overline{\nodal^0_1}(u)$, $\overline{\nodal^0_2}(u)$,  $\overline{\nodal^1_1}(v)$, and $\overline{\nodal^1_2}(v)$ are characterized by the transformations of only a single end point.
Nevertheless, the remaining node functionals $\overline{\nodal^0_3}(u)$, $\overline{\nodal^0_4}(u)$, and $\overline{\nodal^1_3}(v)$ are truly double integrals on the balls around both end points.
\end{remark}

\subsection{Quasi-interpolation operators}
\label{sec:quasi-1d}

We define now the quasi-interpolation operators.
\begin{definition}
Let $\{\overline{\nodal^0_i}\}$, $\{\overline{\nodal^1_i}\}$ be the weighted node functionals as in~\eqref{eq:pi0-1d}, and $\{\phi^{0}_i\}$, $\{\phi^{1}_i\}$ be the basis functions on the reference element $\interval$ as in~\eqref{eq:basis-p2} and~\eqref{eq:basis-p3}. The quasi-interpolation operators on $\interval_\rho$ are defined as:
\begin{gather}
    \label{eq:quasi-interp}
    \begin{aligned}
    \wint_0\colon C^1\Lambda^0(\interval_\rho) &\to \P_3\Lambda^0(\interval)
    \qquad&u &\mapsto \sum_{i=1}^4 \overline{\nodal^0_i}(u) \,\phi^0_i,
    \\
    \wint_1 \colon C^0\Lambda^1(\interval_\rho) &\to \P_2\Lambda^1(\interval)
    &v &\mapsto \sum_{i=1}^3 \overline{\nodal^1_i}(v) \,\phi^1_i,
    \end{aligned}
\end{gather}
Furthermore, for later convenience, we extend both operators such that
\begin{gather}
\label{eq:wint_extension}
    \wint_0\colon C^0\Lambda^1(\interval_\rho) \to 0,
        \qquad
        \wint_1\colon C^1\Lambda^0(\interval_\rho) \to 0.
\end{gather}
\end{definition}

We show now that the quasi-interpolation operators $\wint_0$ and $\wint_1$ are well-defined and bounded on $L^2(\interval_\rho)$.
\begin{theorem}
\label{thm:L2-bound}
    The quasi-interpolation operators admit the following estimates:
    \begin{align*}
        \norm{\wint_0 u}{L^2\Lambda^0(\interval_\rho)}
            \leq C_{\wint_0} \norm{u}{L^2\Lambda^0(\interval_\rho)},
            \quad\forall u\in L^2\Lambda^0(\interval_\rho),\\
        \norm{\wint_1 v}{L^2\Lambda^1(\interval_\rho)}
            \leq C_{\wint_1} \norm{v}{L^2\Lambda^1(\interval_\rho)},
            \quad\forall v\in L^2\Lambda^1(\interval_\rho),   
    \end{align*}
where $C_{\wint_0},\ C_{\wint_1}$ are positive constants depending on $\rho$, but independent of $u$ and $v$.
\end{theorem}

\begin{proof}
It is enough to show that the weighted node functionals $\overline{\nodal^k_i}$, for $k=0,1$, and for all admissible values of $i$,
are bounded on $L^2$.
First, for $u\in L^2(\Omega)$
\begin{align}
    \left|\overline{\nodal^1_1} v\right|
    & = \abs{\int_{B_l} \int_{B_r} \eta_l(\xi_l)\eta_r(\xi_r) v(\xi_l)\,d\xi_r\,d\xi_l}
    = \abs{\int_{B_l} \eta_l(\xi_l)v(\xi_l)\,d\xi_l}\\
    & \leq \norm{\eta_l}{L^2(B_l)} \norm{v}{L^2(B_l)}
    \leq \rho^{-\nicefrac12}\norm{\eta}{L^2(\mathbb R)} \norm{v}{L^2(\interval_\rho)}.
\end{align}
This argument immediately transfer to $\overline{\nodal^1_2}$. For  $\overline{\nodal^0_3}$ (and with appropriate modification for $\overline{\nodal^0_4}$) we obtain by the same means
\begin{align}
    \left|\overline{\nodal^0_3} u\right|
    & = \abs{\int_{B_l} \int_{B_r} \eta_l(\xi_l)\eta_r(\xi_r)
    \bigl(u(\xi_r) - u(\xi_l)\bigr)\,d\xi_r\,d\xi_l}\\
    &= \abs{\int_{B_r} \eta_r(\xi_r)u(\xi_r)\,d\xi_r
    - \int_{B_l} \eta_l(\xi_l)u(\xi_l)\,d\xi_l}\\
    & 
    \leq 2\, \rho^{-\nicefrac12}\norm{\eta}{L^2(\mathbb R)} \norm{u}{L^2(\interval_\rho)}.
\end{align}
For the integral node value we observe
\begin{align*}
    \abs{\overline{\nodal^1_3} v} 
    & = \abs{\int_{B_l} \int_{B_r} \int_{\tilde\interval} \eta_l(\xi_l)\eta_r(\xi_r) 
    v(\tilde x)\dtx\,d\xi_r\,d\xi_l}\\
    & \le \int_{B_l} \int_{B_r} \int_{\tilde\interval} \eta_l(\xi_l)\eta_r(\xi_r)
    \abs{v(\tilde x)}\dtx\,d\xi_r\,d\xi_l\\
    & \le \norm{v}{L^1(\interval_\rho)}
    \int_{B_l} \int_{B_r} \eta_l(\xi_l)\eta_r(\xi_r)\,d\xi_r\,d\xi_l\\
    &\le \sqrt{1+2\rho} \norm{v}{L^2(\interval_\rho)}.
\end{align*}
Finally, we estimate the degrees of freedom involving derivatives using integration by parts, for instance
\begin{align*}
    \abs{\overline{\nodal^0_1} u}
    & = \abs{\int_{B_l}\int_{B_r}\eta_l(\xi_l)\,\eta_r(\xi_r)u'(\xi_l)\,d\xi_r\,d\xi_l}
    = \abs{\int_{B_l}\eta_l(\xi_l)u'(\xi_l)\,d\xi_l}\\
    & = \abs{\int_{B_l}\eta_l'(\xi_l) u(\xi_l)\,d\xi_l}
    \leq \rho^{-\nicefrac32}\norm{\eta'}{L^2(\mathbb R)} \norm{u}{L^2(\interval_\rho)}.
\end{align*}
The statement is then proved, with 

\begin{align*}
    C_{\wint_0} &=
    \rho^{-3/2} \norm{\eta'}{L^2(\R)}
    (\norm{\phi^0_1}{L^2(\interval_\rho)}+\norm{\phi^0_2}{L^2(\interval_\rho)})\\
    &\quad
    + 2\rho^{-1/2}\norm{\eta}{L^2(\R)}
    (\norm{\phi^0_3}{L^2(\interval_\rho)}+\norm{\phi^0_4}{L^2(\interval_\rho)}),\\
    C_{\wint_1} & =
    \rho^{-1/2}\norm{\eta}{L^2(\R)}
    (\norm{\phi^1_1}{L^2(\interval_\rho)}+\norm{\phi^1_2}{L^2(\interval_\rho)})\\
    &\quad + \sqrt{1+2\rho}\norm{\phi^1_3}{L^2(\interval_\rho)}.
\end{align*}
\end{proof}

The following lemma shows that the quasi-interpolation operators $\wint_0$ and $\wint_1$ are co-chain operators.
\begin{lemma}
\label{lem:commuting-quasi-interp}
The exterior derivative and the quasi-interpolation operators in~\eqref{eq:quasi-interp} commute, namely, for all $u\in H\Lambda^0(\interval_\rho)$ there holds
\begin{gather}
    d^0 \wint_0 u = \wint_1 d^0 u.
\end{gather}
\end{lemma}
\begin{proof}
From lemma~\ref{prop:dof_tilde_prop} it follows immediately that for all 
$u\in H\Lambda^0(\interval_\rho)$
\begin{align}
    \label{eq:nodal_prop1}
    \overline{\nodal^1_i}(du)&=\overline{\nodal^0_i}(u), \qquad i=1,2,3.
\end{align}
Thus, the assumptions of lemma~\ref{lemma:commute_nodal} for $\wint_0$ and $\wint_1$ are fulfilled. Since the basis functions are the same as for the canonical interpolation operator $I$, the result follows.
\end{proof}

We note that $\wint_0$ and $\wint_1$ are not interpolation operators in the classical sense. In particular, they do not act as the identity on their range. They share this with the classical quasi-interpolation operators in~\cite{Clement75,ErnGuermond,ScottZhang90} as well as the commuting ones in~\cite{Christiansen07,Schoeberl}. 
By using a trick from~\cite{Schoeberl10}, we can define new operators $\hat\wint_0$ and $\hat\wint_1$ which have the projection property.

To this end, let us first highlight the dependence of $\wint_k$ on the size of the intervals $B_\rho$ by writing $\wint_k(\rho)$ for $k=0,1$. By the definition of the node functionals, we have for any polynomial $p\in\P\Lambda^k$
\begin{gather}
    \wint_k(\rho) p \to \cint_k p \qquad\text{as}\qquad\rho\to 0.
\end{gather}
 Since $\cint_k$ acts as identity on $\P\Lambda^k$ the operator $\wint_k$ is invertible on this space for sufficiently small $\rho$. Then, we can set
 \begin{gather}
     \hat\wint_k = \bigl(\wint_{k|\P\Lambda^k}\bigr)^{-1} \wint_k.
 \end{gather}

\section{Tensor complex in $n$ dimensions}
\label{sec:tensor_complex}

We start the section with a brief introduction to the tensor product of cochain complexes. Then, we detail two particular cases: the tensorization of the $L^2$ de Rahm complex~\eqref{eq:complex_H} on $\interval_\rho$, and the tensorization of the finite element complex on $\interval$ introduced in Section~\ref{sec:1d_complex_I}.

\subsection{Introduction to the tensor product of cochain complexes}

Let $S\subset\R^n$ and $T\subset\R^m$ be two open bounded  domains, and let $V\subset H\Lambda^k(S)$ and $W\subset H\Lambda^\ell(T)$ be Hilbert spaces of differential forms, with inner products $\dual{\cdot}{\cdot}{H\Lambda^k}$ and $\dual{\cdot}{\cdot}{H\Lambda^\ell}$, respectively (see~\eqref{eq:H_inner_prod}).
Following~\cite[Chapter 2]{Reed1980} we recall the definition of the tensor product $V\otimes W$.

Given two differential forms $v\in V$ and $w\in W$, with 
\begin{equation*}
v=\sum_{\sigma\in\Sigma(k,n)}v_\sigma dx^\sigma,\quad
w=\sum_{\tau\in\Sigma(\ell,m)}w_\tau dx^\tau,
\end{equation*}
their tensor product $v\otimes w$ is the $(k+\ell)$-form, expressed in coordinates as
\begin{equation}
\label{eq:tp_forms}
    v\otimes w=\sum_{\substack{\sigma\in\Sigma(k,n)\\\tau\in\Sigma(\ell,m)}} v_\sigma\otimes w_\tau\, dx^\sigma\wedge dx^\tau.
\end{equation}

Denote with $\mathcal{E}$ the set of finite linear combinations of tensor product differential forms as in~\eqref{eq:tp_forms}. We define an inner product on $\mathcal{E}$ as 
\begin{equation}
    \label{eq:tp_inner_pr}
    \dual{v_1\otimes w_1}{v_2\otimes w_2}{H\Lambda^k\otimes H\Lambda^\ell}
    =\dual{v_1}{v_2}{H\Lambda^k} \dual{w_1}{w_2}{H\Lambda^\ell},
\end{equation}
and we extend it by linearity to $\mathcal{E}$.
The tensor product space $V\otimes W$ is the Hilbert space obtained as completition of $\mathcal{E}$ under the inner product~\eqref{eq:tp_inner_pr}. 
If $\{\zeta_i\}$ and $\{\xi_j\}$ are orthonormal bases of the Hilbert spaces $V$ and $W$, respectively, then the set $\{\zeta_i\otimes\xi_j\}$ is an othonormal basis for $V\otimes W$, which we refer to as rank-one basis since it consists of elements of tensor rank one.

Following~\cite{Arnold2015}, we recall the definition of tensor product of complexes of differential forms.
Let there be given two complexes on $S\subset\R^n$ and $T\subset\R^m$
\begin{align}
\arraycolsep1pt
\label{eq:V_complex}
\begin{array}{ccccccccccccc}
    0&\xrightarrow{\;\;\subset\;\;}&\R
    &\xrightarrow{\;\;\subset\;\;}& V^0
    &\xrightarrow{\;\;d\;\;}&V^1
    &\xrightarrow{\;\;d\;\;}&\cdots
    &\xrightarrow{\;\;d\;\;}& V^n
    &\xrightarrow{\;\;d\;\;}&0\\
    0&\xrightarrow{\;\;\subset\;\;}&\R
    &\xrightarrow{\;\;\subset\;\;}& W^0
    &\xrightarrow{\;\;d\;\;}&W^1
    &\xrightarrow{\;\;d\;\;}&\cdots
    &\xrightarrow{\;\;d\;\;}& W^m
    &\xrightarrow{\;\;d\;\;}&0
  \end{array}
\end{align}
which are subcomplexes of the $L^2$ de Rham complex of $S\subset\R^n$ and $T\subset\R^m$, respectively, meaning that $V^k\subset H\Lambda^k(S)$ and $d(V^k)\subset V^{k+1}$ for all $k=0,\ldots,n$, and $W^j\subset H\Lambda^j(T)$ and $d(W^j)\subset W^{j+1}$, for all $j=0,\ldots,m$. The tensor product of the two complexes in~\eqref{eq:V_complex} is the complex
\begin{equation}
    \label{eq:tp_complexes}
    0\xrightarrow{\;\;\subset\;\;}\R
    \xrightarrow{\;\;\subset\;\;}(V\otimes W)^0
    \xrightarrow{\;\;d\;\;}(V\otimes W)^1
    \xrightarrow{\;\;d\;\;}\cdots
    \xrightarrow{\;\;d\;\;}(V\otimes W)^{m+n}
    \xrightarrow{\;\;d\;\;}0,
\end{equation}
where the space $(V\otimes W)^k$ is defined as
\begin{equation}
    \label{eq:tp_spaces}
    (V\otimes W)^k:=\bigoplus_{i+j=k} \left(V^i\otimes W^j\right),\quad k=0,\ldots,m+n,
\end{equation}
and the exterior derivative $d:(V\otimes W)^k\rightarrow (V\otimes W)^{k+1}$ is defined as
\begin{gather}
    \label{eq:d-tensor}
    d^k(u\otimes v) = d^i u\otimes v + (-1)^i u \otimes d^j v,
    \quad u\in V^i,\ v\in W^j,\ i+j=k.
\end{gather}
Note that the complex~\eqref{eq:tp_complexes} is a subcomplex of the de Rham complex on the Cartesian product $S\times T$.
This construction generalizes to the tensor product of any finite number of subcomplexes of the $L^2$ de Rham complex.

\subsection{Tensorization of the $L^2$ de Rham complex}

We detail now the particular case where $S=T=\interval_\rho\subset\R$, and the complexes in~\eqref{eq:V_complex} coincide with the $L^2$ de Rham complex on $\interval_\rho$  (see~\eqref{eq:complex_H})
\begin{equation}
    \label{eq:complex_H_I}
    0\xrightarrow{\;\;\subset\;\;}
    \R\xrightarrow{\;\;\subset\;\;} H\Lambda^0(\interval_\rho)
    \xrightarrow{\;\; d\;\;} H\Lambda^1(\interval_\rho)
    \xrightarrow{\;\; d\;\;} 0.
\end{equation}
The tensorization of the complex~\eqref{eq:complex_H_I} with itself gives the following complex on the square $\interval_\rho\times\interval_\rho$
\begin{equation}
    \label{eq:complex_H_tens2}
    0\xrightarrow{\;\;\subset\;\;}\R
    \xrightarrow{\;\;\subset\;\;}\complexH{0}
    \xrightarrow{\;\;d\;\;}\complexH{1}
    \xrightarrow{\;\;d\;\;}\complexH{2}
    \xrightarrow{\;\;d\;\;}0,
\end{equation}
where, in accordance with~\eqref{eq:tp_spaces}, we have
\begin{equation}
    \label{eq:H_tens2}
    \complexH{k}:=\bigoplus_{\substack{i+j=k\\i,j=0,1}} H\Lambda^i\otimes H\Lambda^j,
    \quad k=0,1,2.
\end{equation}
In particular, we have
\begin{equation}
\begin{aligned}
    \complexH{0} & = H\Lambda^0 \otimes H\Lambda^0,\\
    \complexH{1} & = \left(H\Lambda^0 \otimes H\Lambda^1\right)\oplus\left(H\Lambda^1 \otimes H\Lambda^0\right),\\
    \complexH{2} & = H\Lambda^1 \otimes H\Lambda^1.
\end{aligned}
\end{equation}
The exterior derivative $d:\complexH{k}\rightarrow\complexH{k+1}$ is defined in equation~\eqref{eq:d-tensor}. In particular, it holds
\begin{equation}
\label{eq:d_tens2}
\begin{aligned}
    d^0 (u_0\otimes v_0) &= d^0 u_0\otimes v_0 + u_0\otimes d^0 v_0,\quad& \forall\, u_0\otimes v_0&\in H\Lambda^0\otimes H\Lambda^0,\\
    d^1 (u_0\otimes v_1) &= d^0 u_0\otimes v,
    & \forall\, u_0\otimes v_1&\in H\Lambda^0\otimes H\Lambda^1,\\
    d^1 (u_1\otimes v_0) &= u_1\otimes d^0 v_0,
    & \forall\, u_1\otimes v_0&\in H\Lambda^1\otimes H\Lambda^0.
\end{aligned}
\end{equation}

Formula~\eqref{eq:complex_H_tens2} and~\eqref{eq:H_tens2} generalize to the $n$-fold tensor product, leading to the following complex on the $n$-dimensional hypercube $\interval_\rho^{\times n}$
\begin{equation}
    \label{eq:complex_H_tensn}
    0\xrightarrow{\;\;\subset\;\;}\R
    \xrightarrow{\;\;d\;\;}\complexHtens{n}{0}
    \xrightarrow{\;\;d\;\;}\complexHtens{n}{1}
    \xrightarrow{\;\;d\;\;}\cdots
    \xrightarrow{\;\;d\;\;}\complexHtens{n}{n}
    \xrightarrow{\;\;d\;\;}0.
\end{equation}
The space $\complexHtens{n}{k}$, for $k=0,\ldots,n$, is defined as
\begin{equation}
    \label{eq:H_tensk}
    \complexHtens{n}{k}:=\bigoplus_{\mathbf{i}\in\chi_k} H\Lambda^{i_1}\otimes\cdots\otimes H\Lambda^{i_n}.
\end{equation}
Here we employ an alternative representation of $\Sigma(k,n)$ by characteristic vectors,
where the binary vector $\mathbf i$ selects $k$ out of the $n$ fibers, and is thus taken from the set
\begin{gather}
    \label{eq:chi}
    \chi_k:=\left\{\boldsymbol{\ell}=(\ell_1,\ldots,\ell_n)\in\{0,1\}^n
    \;\middle|\;\sum_{j=1}^n\ell_j=k\right\}.
\end{gather}
Note that the tensor product space $\complexHtens{n}{k}$ is a proper, dense subspace of the space $H\Lambda^k(\interval_\rho^{\times n})$. For $H^1 = H\Lambda^0$, see~\cite[Section 3.4.2]{Hackbusch14}, for the other spaces, note that $\complextens{C^\infty}{n}{k}$ is dense in $\complexHtens{n}{k}$ as well as in $H\Lambda^k(\interval_\rho^{\times n})$.

By straightforward computations, and making use of~\eqref{eq:d-tensor}, we derive the following formula for the exterior derivative $d:\complexHtens{n}{k}\rightarrow\complexHtens{n}{k+1}$:
\begin{equation}
\label{eq:d-tensor-k}
    d^k(u_1\otimes \cdots\otimes u_n)= 
    \sum_{j=1}^n \theta_j\,
    (u_1\otimes\cdots\otimes d u_j\otimes\cdots\otimes u_n),
\end{equation}
where $\theta_j\in\{-1,1\}$ is defined as $\theta_j:=(-1)^{\sum_{\ell=1}^{j-1} i_\ell}$.

\subsection{Tensorization of the finite element complex}

We focus now on the finite element complex 
\begin{equation}
    \label{eq:1d_complexP}
    0\xrightarrow{\;\;\subset\;\;}\R
    \xrightarrow{\;\;\subset\;\;}\P_3\Lambda^0(\interval) \xrightarrow{\;\;d\;\;}
    \P_2\Lambda^1(\interval)
    \xrightarrow{\;\;d\;\;}0
\end{equation}
with node functionals as in~\eqref{eq:pi0-1d}.
We write the generic $\P\Lambda$ to refer to either space of polynomial forms, with the understanding that $\P\Lambda^0 \equiv \P_3\Lambda^0$ and $\P\Lambda^1 \equiv \P_2\Lambda^1$.

Applying the tensor product construction to~\eqref{eq:1d_complexP}, we find the following tensor product complex on $\interval^{\times n}$
\begin{gather}
\label{eq:complex_P_tens}
    0 \xrightarrow{\;\;\subset\;\;}\R 
    \xrightarrow{\;\;\subset\;\;}\complexPtens{n}{0} \xrightarrow{\;\;d\;\;} \complexPtens{n}{1} \xrightarrow{\; d\;}\cdots \complexPtens{n}{n} \xrightarrow{\;\;d\;\;} 0,
\end{gather}
where the space $\complexPtens{n}{k}$, for $k=0,\ldots,n$, is defined as
\begin{gather}
    \label{eq:Ptens}
    \complexPtens{n}{k} 
    = \bigoplus_{\mathbf i\in\chi_k}
    \P\Lambda^{i_1}\otimes\cdots\otimes\P\Lambda^{i_n},
\end{gather}
the set $\chi_k$ being introduced in~\eqref{eq:chi},
and the exterior derivative $d:\complexPtens{n}{k}\rightarrow\complexPtens{n}{k+1}$ is as in~\eqref{eq:d-tensor-k}.

The tensor product construction yields node functionals
for $\complexPtens{n}{k}$ of the form $\nodal_1\otimes\cdots\otimes\nodal_n$, and is associated with the Cartesian product $f_1\times\cdots\times f_n$, where $\nodal_j$ is the weighted node functional as in~\eqref{eq:pi0-1d} associated to $f_j$ ($f_j$ is either a vertex of $\interval_\rho$ or $\interval_\rho$ itself). Then, \eqref{eq:complex_P_tens} is a complex of finite element differential forms.

\section{Tensor product of quasi-interpolation operators}
\label{sec:tensor-quasi-interpolation}

\subsection{Introduction to the tensor product of operators on Hilbert spaces}

We start recalling the definition and some properties of the tensor product of operators on Hilbert spaces (see~\cite{Reed1980}).
\begin{definition}
\label{def:tp_operators}
Let $V,\,W$ be two Hilbert spaces, and let $F\colon V\rightarrow V^\prime$ and $G\colon W\rightarrow W^\prime$ be continuous operators. The tensor product operator $F\otimes G\colon V\otimes W\rightarrow V^\prime\otimes W^\prime$ is defined on functions of the type $v\otimes w$ as
\begin{equation*}
    (F\otimes G)(v\otimes w)=F(v) \otimes G(w),
\end{equation*}
and is then extended by linearity and density. 
\end{definition}
In~\cite[Chapter 8]{Reed1980} the authors prove the following result.
\begin{lemma}
\label{lem:tp_bounded_op}
Let the spaces $V,\, W$ and the operators $F,\, G$ be as in Definition~\ref{def:tp_operators}. Then, the tensor product operator $F\otimes G$ is bounded. In particular, it holds
\begin{equation}
    \norm{F\otimes G}{\mathcal L(V\otimes W,V^\prime\otimes W^\prime)}
    =\norm{F}{\mathcal L(V,V^\prime)}
    \norm{G}{\mathcal L(W,W^\prime)}.
\end{equation}
\end{lemma}
Definition~\ref{def:tp_operators} and Lemma~\ref{lem:tp_bounded_op} generalize to the tensor product of any finite number of bounded operators on Hilbert spaces.

\subsection{Commuting quasi-interpolation operators in $n$ dimensions}

Let us take $F=G=\wint_0$, where $\wint_0$ is the commuting quasi-interpolation operator in one dimension from Section~\ref{sec:quasi-1d}. Using Definition~\ref{def:tp_operators} and the Riesz representation theorem, we define the tensor product operator 
\begin{equation*}
    \wint_0\otimes\wint_0\colon L^2\Lambda^0\otimes L^2\Lambda^0 
    \rightarrow \P_3\Lambda^0\otimes\P_3\Lambda^0.
\end{equation*}
In the same way, we define 
\begin{align*}
    \wint_0\otimes\wint_1&\colon L^2\Lambda^0\otimes L^2\Lambda^1 
    \rightarrow \P_3\Lambda^0\otimes\P_2\Lambda^1,\\
    \wint_1\otimes\wint_0&\colon L^2\Lambda^1\otimes L^2\Lambda^0 
    \rightarrow \P_2\Lambda^1\otimes\P_3\Lambda^0,\\
    \wint_1\otimes\wint_1&\colon L^2\Lambda^1\otimes L^2\Lambda^1
    \rightarrow \P_2\Lambda^1\otimes\P_2\Lambda^1.
\end{align*}

\begin{definition}
\label{def:tp_pi}
Given the commuting quasi-interpolation operators in one dimension from section~\ref{sec:quasi-1d}, we define the tensor product quasi-interpolator in two dimensions for $k=0,1,2$, namely
\begin{gather*}
      \winttens{2}{k}:\complexLtens{2}{k}\rightarrow\complexPtens{2}{k}
\end{gather*}
by
\begin{equation}
\label{eq:wint2}
\begin{aligned}
    \winttens{2}{0}&=\wint_0\otimes \wint_0,\\
    \winttens{2}{1}&=\wint_0\otimes \wint_1 + \wint_1\otimes \wint_0,\\
    \winttens{2}{2}&=\wint_1\otimes \wint_1,
\end{aligned}
\end{equation}
where the space $\complexPtens{2}{k}$ has been defined in~\eqref{eq:Ptens}, and the space $\complexLtens{2}{k}$ is defined as
\begin{gather*}
    \complexLtens{2}{k}
    =\bigoplus_{\substack{i+j=k\\i,j=0,1}} L^2\Lambda^{i}\otimes L^2\Lambda^{j},
    \quad k=0,1,2.
\end{gather*}
\end{definition}
Using the extension by zero in~\eqref{eq:wint_extension}, we can write
\begin{equation}
  \label{eq:wint_tp2}
  \winttens{2}{k}:=\sum_{\substack{i+j=k\\i,j=0, 1}} \wint_i\otimes\wint_j,
\end{equation}
since for $u\in L^2\Lambda^{i'}$, $v\in L^2\Lambda^{j'}$, with $i'+j'=k$, 
we have
\begin{equation}
    \label{eq:wint_tens2}
    \winttens{2}{k}(u\otimes v) =
    \sum_{\substack{i+j=k\\i,j=0,1}}
    \wint_{i}(u)\otimes \wint_{j}(v)
    = \wint_{i'}(u)\otimes\wint_{j'}(v),
\end{equation}
Definition~\eqref{eq:wint_tens2} extends by linearity and density to all elements of $\complexLtens{2}{k}$.

The definition of interpolation operators in the form~\eqref{eq:wint_tp2} generalizes to the tensor product of any finite number of quasi-interpolation operators by the following construction. We start defining the domain of this tensor product operator in $n$ dimensions:
\begin{gather*}
    \complexLtens{n}{k}
    =\bigoplus_{\mathbf{i}\in\chi_k} L^2\Lambda^{i_1}\otimes\cdots\otimes L^2\Lambda^{i_n},
    \quad k=0,\ldots,n.
\end{gather*}
\begin{remark}
Note that, by Fubini's theorem, the following isomorphisms hold:
\begin{align}
    \label{eq:L20}
    \complexLtens{n}{0}(\interval_\rho^{\times n}) &= \underbrace{L^2\Lambda^0(\interval_\rho) \otimes\cdots\otimes L^2\Lambda^0(\interval_\rho)}_{n\ \text{times}}\simeq L^2\Lambda^0(\interval_\rho^{\times n}),\\
    \label{eq:L2n}
    \complexLtens{n}{n}(\interval_\rho^{\times n})& =\underbrace{L^2\Lambda^1(\interval_\rho) \otimes\cdots\otimes L^2\Lambda^1(\interval_\rho)}_{n\ \text{times}}\simeq L^2\Lambda^n(\interval_\rho^{\times n}).
\end{align}
\end{remark}

\begin{definition}
Given the commuting quasi-interpolation operators in one dimension from section~\ref{sec:quasi-1d}, we define the tensor product quasi-interpolator in $n$ dimensions, for $k=0,1,\ldots,n$, as
\begin{equation}
  \label{eq:wint_tpn}
  \winttens{n}{k}\colon\complexLtens{n}{k}\rightarrow\complexPtens{n}{k},
  \qquad
  \winttens{n}{k}:= \sum_{\mathbf{i}\in\chi_k}\wint_{i_1}\otimes\cdots\otimes\wint_{i_n}.
\end{equation}
\end{definition}

The operator $\winttens{n}{k}$ 
applies to the tensor product of rank-one functions as follows: given $u_1\otimes\cdots\otimes u_n\in \complexLtens{n}{k}$, with $u_j\in L^2\Lambda^{i_j}$ and $\mathbf{i}=(i_1,\ldots,i_n)\in\chi_k$, it holds
\begin{equation}
\label{eq:wint_tensk}
    \winttens{n}{k}(u_1\otimes\cdots\otimes u_n)
    = \sum_{\substack{\mathbf{i}'\in\chi_k}}
    \wint_{i'_1}(u_1)\otimes\cdots\otimes\wint_{i'_n}(u_n)
    =\wint_{i_1}(u_1)\otimes\cdots\otimes\wint_{i_n}(u_n),
\end{equation}
where in the second equality we have used~\eqref{eq:wint_extension}.
Definition~\eqref{eq:wint_tensk} extends by linearity and density to all elements of $\complexLtens{n}{k}$.

\begin{lemma}
\label{lem:L2-bound-tens}
The quasi-interpolation operator $\winttens{n}{k}$ defined in~\eqref{eq:wint_tpn} is bounded in $L^2(\interval_\rho^{\times n})$.
\end{lemma}
\begin{proof}
Lemma~\ref{lem:tp_bounded_op} states that the tensor product of bounded operators on Hilbert spaces is bounded, with constant given as product of the individual constants. Then, it holds:
\begin{align*}
    \norm{\winttens{n}{k}}{}
    & \leq \sum_{\mathbf{i}\in\chi_k} \norm{\wint_{i_1}\otimes\cdots\otimes\wint_{i_n}}{}
    = \sum_{\mathbf{i}\in\chi_k} \norm{\wint_{i_1}}{}\cdots\norm{\wint_{i_n}}{}\\
    & \leq \left(\sum_{\mathbf{i}\in\chi_k} 1\right) C_\wint^n
    = \bc{n}{k} C_\wint^n,
\end{align*}
where $C_\wint=\max\{C_{\wint_0},C_{\wint_1}\}$, the constants $C_{\wint_0}, C_{\wint_1}$ being introduced in Theorem~\ref{thm:L2-bound}.
\end{proof}

The following lemma shows that $\winttens{n}{k}$ is a co-chain operator.
\begin{lemma}
The tensor product operator $\winttens{n}{k}$ commutes with the exterior derivative.
More precisely, for $u\in (H\Lambda^{\otimes n})^k$, there holds
\begin{gather}
    \winttens{n}{k+1}(d^k u)
    = d^k\winttens{n}{k} u
\end{gather}
\end{lemma}

\begin{proof}
We start proving the result on rank-one functions $u_1\otimes\cdots\otimes u_k\in\complexLtens{n}{k}$, with $u_j\in L^2\Lambda^{i_j}$, $\mathbf{i}=(i_1,\ldots,i_n)\in\chi_k$. Using~\eqref{eq:d-tensor-k}, the linearity of $\winttens{n}{k}$, ~\eqref{eq:wint_extension} and Lemma~\ref{lem:commuting-quasi-interp}, there holds
\begin{align*}
    \winttens{n}{k+1}(d^k(u_1\otimes\cdots\otimes u_n))
    & = \winttens{n}{k+1} \left(\sum_{j=1}^n\theta_j\, u_1\otimes\cdots\otimes d^{i_j}u_j\otimes\cdots\otimes u_n\right)\\
    & = \sum_{j=1}^n\theta_j \winttens{n}{k+1}(u_1\otimes\cdots\otimes d^{i_j}u_j\otimes\cdots\otimes u_n)\\
    & = \sum_{j=1}^n\theta_j \wint_{i_1}(u_1)\otimes\cdots\otimes\wint_{i_j+1}(d^{i_j}u_j)\otimes\cdots\otimes\wint_{i_n}(u_n)\\
    & = \sum_{j=1}^n\theta_j \wint_{i_1}(u_1)\otimes\cdots\otimes d^{i_j}\wint_{i_j}(u_j)\otimes\cdots\otimes\wint_{i_n}(u_n)\\
    & = d^k(\winttens{n}{k}(u_1\otimes\cdots\otimes u_n)).
\end{align*}
The result extends by linearity and density to all elements of the tensor product space $\complexLtens{n}{k}$.
\end{proof}

\begin{remark}
A similar tensor product construction has been applied to bounded cochain projectors in~\cite{Bonizzoni2013}.
\end{remark}

\section{Higher-order polynomial spaces}
\label{sec:higher-order}

In this section, we repeat the construction of sections~\ref{sec:1d_complex_I} and~\ref{sec:quasi_interpolation} to obtain commuting quasi-interpolation operators for finite elements of arbitrary polynomial order.

\subsection{Canonical interpolation operators}

The canonical commuting interpolation operators in section~\ref{sec:1d_complex_I} extend to higher order polynomial spaces in a straight-forward way,
if we introduce additional node functionals 
and corresponding basis functions.
To this end, let $\ell_m \in \P_m$ be the Legendre polynomial of degree $m$ on the interval $\interval$, normalized such that $\ell_m(1) = 1$. Then, the sequence $\{\ell_m\}_{m=0,\dots}$ is mutually $L^2(\interval)$-orthogonal. 
We also introduce the integrated and twice integrated Legendre polynomials
\begin{gather}
        L_m(x) = \int_0^x \ell_m(t)\,dt,
        \qquad
        K_m(x) = \int_0^x L_m(t)\,dt.
\end{gather}

We recall the well known relation 
\begin{gather}
    \label{eq:Legendre-integral}
    2(2m+1) L_m(x) = \ell_{m+1}(x) - \ell_{m-1}(x),
\end{gather}
which implies the following properties:
\begin{enumerate}[(i)]
    \item $L_m(0) = L_m(1) = 0$ and equivalently $K^\prime_m(0) = K^\prime_m(1) = 0$ for $m\geq 1$, since 
    $L_m$ is the difference of two Legendre polynomials of equal parity.
    \item $K_m(0) = K_m(1) = 0$, for $m\ge 2$, since $K_m$ ($m\ge 2$) is the integral of a function with zero mean vanishing at the interval ends.
\end{enumerate}

We define the node functionals for higher-order polynomial finite elements  $\P_m\Lambda^0$ and $\P_{m-1}\Lambda^1$ by the following interpolation conditions.
\begin{itemize}
    \item For $\P_m\Lambda^0(\interval)$, use
    \begin{gather}
        \label{eq:i0-1d-k}
        \begin{aligned}
            \nodal^0_1 (u) &= u'(0) & \qquad\nodal^0_{m+1} (u) &= u(0) + u(1)\\
            \nodal^0_2 (u) &= u'(1)\\
            \nodal^0_{i+3}(u) &= \int_{\interval} \ell_i u^\prime \dx,
        & i&=0,\dots,m-3.
        \end{aligned}
    \end{gather}
    While $\nodal^0_1$ and $\nodal^0_2$ are identical to~\eqref{eq:i0-1d}, we replaced $\nodal^0_3$ by an integral over $u'$, which evaluates to the same as the original. The functional $\nodal^0_4$ stayed the same, but now received the index $m+1$ to be conforming with lemma~\ref{lemma:commute_nodal}.
    In what follows, we will refer to $\{\nodal^0_1,\nodal^0_2,\nodal^0_3,\nodal^0_{m+1}\}$ as the original node functionals and introduce the corresponding index set $J_o = \{1,2,3,m+1\}$.
    
    \item For $\P_{m-1}\Lambda^1(\interval)$, introduce the node functionals
    \begin{gather}
        \label{eq:i1-1d-k}
         \begin{split}
            \nodal^1_1 (v) &= v(0)\\
            \nodal^1_2 (v) &= v(1)
        \end{split}
        \qquad \nodal^1_{i+3} (v) = \int_{\interval} \ell_i v\dx
        \qquad i=0,\dots,m-3.
    \end{gather}
\end{itemize}

Moreover, we choose the basis $\{\phi^0_i\}$ and $\{\phi^1_i\}$ for the spaces $\P_m=\P_m\Lambda^0(\interval)$ and $\P_{m-1}=\P_{m-1}\Lambda^1(\interval)$, respectively, as follows:
\begin{itemize}
    \item $\phi^0_1,\ \phi^0_2,\ \phi^0_3$ and $\phi^0_{m+1}$ are chosen identical to~\eqref{eq:basis-p3}. The remaining polynomials are chosen as 
    \begin{gather}
        \label{eq:i0-k-basis}
        \phi^0_i(x) = K_{i-2}(x),\qquad i=4,\dots,m,
    \end{gather}
    where we note that $K_{i-2}$ has degree $i$.
    \item $\phi^1_1,\ \phi^1_2$ and $\phi^1_3$ are chosen identical to~\eqref{eq:basis-p2}. The remaining polynomials are chosen as 
    \begin{gather}
        \label{eq:i1-k-basis}
        \phi^1_i(x) = L_{i-2}(x),\qquad i=4,\dots,m,
    \end{gather}
    where we note that $L_{i-2}$ has degree $i-1$.
\end{itemize}

The canonical interpolation operators, as before $\cint_0\colon \Lambda^0(\interval)\rightarrow\P_m\Lambda^0(\interval)$ and $\cint_1\colon \Lambda^1(\interval)\rightarrow\P_{m-1}\Lambda^1(\interval)$ are then defined as
\begin{gather}
    \label{eq:cint-1d-k}
    \cint_0 u(x) = \sum_{i=1}^{m+1} \nodal^0_i(u) \,\phi^0_i(x),
    \qquad \cint_1 v(x) = \sum_{j=1}^m \nodal^1_j (v) \,\phi^1_j(x).
\end{gather}

\begin{lemma}
    The space $\P_m=\P_m\Lambda^0(\interval)$ with the node functionals in~\eqref{eq:i0-1d-k} forms a unisolvent finite element.
\end{lemma}
\begin{proof}
 First, we note that the dimension of $\P_m$ equals the number of node functionals. Thus, it is sufficient to show that for $p\in \P_m$ there holds
 \begin{equation}
    \label{eq:unisolvence-pk}
    \biggl[\quad\nodal^0_i(p) = 0\quad \forall\,i = 1,\dots,m+1\quad\biggr]
    \quad\Longrightarrow\quad
    \Bigl[\;p\equiv 0\;\Bigr].
 \end{equation}

We show~\eqref{eq:unisolvence-pk} by writing $p$ as linear combination of the basis $\{\phi^0_i\}$:
 \begin{gather*}
     p(x) = \sum_{i=1}^{m+1} \alpha_i \phi^0_i(x).
 \end{gather*}
Since $K_i$ for $i\geq 2$ has double roots at 0 and 1, there holds
 \begin{gather}
     \nodal^0_i(\phi^0_j) = 0 \qquad i\in J_o, \quad j=4,\dots,m.
 \end{gather}
 Therefore, by standard Hermitian interpolation conditions, we obtain
 \begin{gather}
     \Bigl[\;\nodal^0_i(p) = 0 \quad\forall i\in J_o\;\Bigr]
     \quad \Longrightarrow \quad
     \Bigl[\;\alpha_i = 0\quad\forall i\in J_o\;\Bigr].
 \end{gather}
 
 For the remaining coefficients we prove $\alpha_i = 0$ by induction. First note for $i=1,\dots,m-3$
 \begin{align*}
    \nodal^0_{i+3}(p)
    &= \sum_{j=4}^{m} \int_{\interval} \ell_i (x) \alpha_j K^\prime_{j-2}(x)\dx
    = \sum_{j=4}^{m} \int_{\interval} \ell_i (x) \alpha_j L_{j-2}(x)\dx\\
    &\stackrel{\eqref{eq:Legendre-integral}}{=} 
    \sum_{j=4}^{m} 
    \tfrac{\alpha_j}{2(2j-3)}
    \int_{\interval} \ell_i (x)  \bigl(\ell_{j-1}(x) - \ell_{j-3}(x)\bigr)\dx.
\end{align*}

Thus, by orthogonality of the Legendre polynomials
\begin{gather}
    \nodal^0_4(p) = \tfrac{\alpha_4}{10}
    \int_{\interval} \ell_1^2 (x) \dx.
\end{gather}
We conclude that $\nodal^0_4(p)=0$ implies $\alpha_4 = 0$. Assume now that 
$4<n<m$, and $\alpha_k=0$ for all $1 \le k \le n-1$.
Then,
\begin{align}
    \nodal^0_{n}(p)
    &= \sum_{j=n}^{m} 
    \alpha_j\tfrac1{2(2j-3)}
    \int_{\interval} \ell_{n-3} (x)  \bigl(\ell_{j-1}(x) - \ell_{j-3}(x)\bigr)\dx\\
    &= \tfrac{\alpha_n}{2(2n-3)} 
    \int_{\interval}  \ell_{n-3}^2(x) \dx.
\end{align}
Hence, $\nodal^0_n(p)=0$ implies $\alpha_n = 0$.
\end{proof}

By a similar, but simpler argument, we can prove
\begin{lemma}
    The space $\P_{m-1}=\P_{m-1}\Lambda^1(\interval)$ with the node functionals in~\eqref{eq:i1-1d-k} forms a unisolvent finite element.
\end{lemma}

Analog to lemma~\ref{lemma:commuting-1d}, we have the following lemma.
\begin{lemma}
\label{lemma:commuting-1d-k}
The following diagram commutes:
\begin{equation}
\label{eq:diagram-k}
\begin{tikzcd}[column sep=6pc]
C^1\Lambda^0(\interval) \arrow{r}{d} \arrow{d}{I} & 
  C^0\Lambda^1(\interval) \arrow{d}{I} \\
\P_m\Lambda^0(\interval) \arrow{r}{d} &
  \P_{m-1}\Lambda^1(\interval)
\end{tikzcd}
\end{equation} that is,
for every $u\in C^1\Lambda^0(\interval)$, there holds
\begin{gather}
    \label{eq:commute-1d-k}
    d^0 \cint_0 u = \cint_1 d^0 u.
\end{gather}
\end{lemma}

\begin{proof}
We show that lemma~\ref{lemma:commute_nodal} applies to the 
node functionals~\eqref{eq:i0-1d-k} and~\eqref{eq:i1-1d-k}, and the basis functions~\eqref{eq:i0-k-basis} and~\eqref{eq:i1-1d-k}.
 Since the functionals $\nodal^0_i$ with $i\in J_o$ and $\nodal^1_i$ with $i\in\{1,2,3\}$ have not changed except for the reformulation of $\nodal^0_3$, the result of lemma~\ref{lemma:commuting-1d} still applies to those. For the remaining ones, we have
\begin{gather}
    \nodal^1_{i+3}(du) = \int_{\interval} \ell_i du\dx
    = \int_{\interval} \ell_i u'\dx = \nodal^0_{i+3}(u),
    \qquad i=1,\dots,m-3.
\end{gather}
We conclude observing that, by definition, $d\phi^0_i=\phi^1_i$ for $i=1,\ldots,m$, and $d\phi^0_{m+1}=0$ holds.
\end{proof}

\subsection{Quasi-interpolation operators}
\label{ssec:quasi_interpolation_k}

Again, we introduce node functionals
on the perturbed interval $\tilde\interval_{y_l,y_r}$.
To this end, it is sufficient to define transformed versions of the new node functionals in~\eqref{eq:i0-1d-k} and~\eqref{eq:i1-1d-k}, since the original ones with index in $J_o$ are transformed as before 
(see equations~\eqref{eq:node-tilde} and~\eqref{eq:node1-tilde}).
To this end, let $\{\tilde \ell_i\}$ be the sequence of orthogonal polynomials on $\tilde\interval_{y_l,y_r}$, normalized such that $\tilde\ell_i(y_r) = 1$. The transformed node functionals are:
\begin{align}
    \widetilde{\nodal^0_{i+3}}(u)
    &= \int_{\tilde\interval_{y_l,y_r}} \tilde\ell_i(\tilde x) u'(\tilde x)\dtx,
    \\
    \widetilde{\nodal^1_{i+3}}(u)
    &= \int_{\tilde\interval_{y_l,y_r}} \tilde\ell_i(\tilde x) u(\tilde x)\dtx.
\end{align}
Note that by this definition, we still have the commutation property
\begin{gather}
    \widetilde{\nodal^1_{i+3}}(d u) = \widetilde{\nodal^0_{i+3}}(u)
    \qquad i=0,\dots,m-3.
\end{gather}

We can now define the 
weighted node functionals as in~\eqref{eq:pi0-1d}, and the quasi-interpolation operator as in~\eqref{eq:quasi-interp}. Lemma~\ref{lemma:commute_nodal} applies, hence the quasi-interpolation operators commute with the exterior derivative. Moreover, the tensor product construction of Section~\ref{sec:tensor_complex} leads to commuting quasi-interpolation operators with values in $(\P\Lambda^{\otimes n})^k$ as in~\eqref{eq:Ptens}, where $\P\Lambda^0=\P_m\Lambda^0(\interval)$ and $\P\Lambda^1=\P_{m-1}\Lambda^1(\interval)$.
 
\section{Conclusions}

The one-dimensional $H^1$-conforming finite element cochain complex based on cubic and higher-order polynomials, and its quasi-interpolation operators were introduced. The tensor product construction was employed to derive (i) $H^1$-conforming finite element cochain complexes on meshes with Cartesian mesh cells of arbitrary dimension; (ii) $L^2$-stable quasi-interpolation cochain operators. 

The construction principle in section~\ref{sec:higher-order} can be generalized to higher differentiability by adding more derivative degrees of freedom at the interval edges and adjusting the remaining degrees of freedom. The argument using integration by parts in theorem~\ref{thm:L2-bound} remains valid if applied multiple times and yields $L^2$-stable quasi-interpolation operators also for this case.

When we refer to meshes with Cartesian mesh cells, we mean that all boundaries of mesh cells are axiparallel. While this is more general than a Cartesian mesh and allows for domains with nontrivial topology, it is nevertheless very restrictive. Lifting this condition is not trivial though. First, it is known that the relation $dV^k\subset V^{k+1}$ does not hold anymore directly for finite element spaces, but only after applying an additional Riesz isomorphism, see~\cite{ArnoldBoffiFalk05}. Second, the tensor product construction of degrees of freedom requires coordinate systems in vertices, which are consistent over all attached cells. This can be achieved at ``regular vertices'', see~\cite{ArndtKanschat}, but it is not clear, whether a construction at irregular vertices can be obtained.


\end{document}